\documentclass[12pt,a4paper]{article}
\usepackage{amssymb,amsmath,amsthm}
\usepackage{graphicx}
\usepackage{hyperref}

\theoremstyle{plain}
\newtheorem{theorem}{Theorem}[section]
\newtheorem{lemma}[theorem]{Lemma}
\newtheorem{corollary}[theorem]{Corollary}

\newtheorem{proposition}[theorem]{Proposition}
\newtheorem{observation}[theorem]{Observation}
\newtheorem{question}[theorem]{Question}

\theoremstyle{definition}

\newtheorem*{rem}{Remark}

\newcommand\cref[1]{Corollary~\ref{cor:#1}}

\title{Adaptive Majority Problems\\ for Restricted Query Graphs\\ and for Weighted Sets\thanks{Research supported by the Lend\"ulet program of the Hungarian Academy of Sciences (MTA), under grant number LP2017-19/2017, the J\'anos Bolyai Research Fellowship of the Hungarian Academy of Sciences, the National Research, Development and Innovation Office -- NKFIH under the grants K 116769, K 124171, K 132696, SNN 129364, KH 130371 and FK 132060, the National Research, Development and Innovation Fund (TUDFO / 51757 / 2019-ITM, Thematic Excellence Program), by the European Union, co-financed by the European Social Fund, by the grant of Russian Government N 075-15-2019-1926, by the New National Excellence Program under the grant number \'UNKP-19-4-BME-287 and  by the  BME-Artificial Intelligence FIKP grant of EMMI (BME FIKP-MI/SC).}}

\begin{document}

\author{G\'abor Dam\'asdi$^{c}$,
D\'aniel Gerbner$^{a}$,
Gyula O.H. Katona$^{a}$,
Bal\'azs Keszegh$^{a,c}$,\\
D\'aniel Lenger$^{c}$,
Abhishek Methuku$^{b}$,
D\'aniel T. Nagy$^{a}$,\\ D\"om\"ot\"or P\'alv\"olgyi$^{c}$,
Bal\'azs Patk\'os$^{a,d}$,
M\'at\'e Vizer$^{a}$,
G\'abor Wiener$^{e}$
\\
\small $^a$ Alfr\'ed R\'enyi Institute of Mathematics, Hungarian Academy of Sciences\\
\small $^b$ \'Ecole Polytechnique F\'ed\'erale de Lausanne\\
\small $^c$ MTA-ELTE Lend\"ulet Combinatorial Geometry Research Group,  E\"otv\"os Lor\'and University\\
\small $^d$  Lab. of Combinatorial and Geometric Structures, Moscow Inst. of Physics and Technology \\
\small $^e$ Dept. of Computer Science and Information Theory, Budapest Univ. of Technology and Economics
}
\maketitle

\begin{abstract}
Suppose that the vertices of a graph $G$ 
are colored with two colors in an unknown way.
The color that occurs on more than half of the vertices is called the \emph{majority color} (if it exists), and any vertex of this color is called a \emph{majority vertex}.
We study the problem of finding a majority vertex (or show that none exists), if we can query edges to learn whether their endpoints have the same or different colors.
Denote the least number of queries needed in the worst case by $m(G)$.
It was shown by Saks and Werman that $m(K_n)=n-b(n)$, where $b(n)$ is the number of 1's in the binary representation of $n$.

In this paper we initiate the study of the problem for general graphs.
The obvious bounds for a connected graph $G$ on  $n$ vertices are $n-b(n)\le m(G)\le n-1$.
We show that for any tree $T$ on an even number of vertices we have $m(T)=n-1$, and that for any tree $T$ on an odd number of vertices, we have $n-65\le m(T)\le n-2$. 
Our proof uses results about the weighted version of the problem for $K_n$, which may be of independent interest.
We also exhibit a sequence $G_n$ of graphs with $m(G_n)=n-b(n)$ such that $G_n$ has $O(nb(n))$ edges and $n$ vertices.
\end{abstract}

\section{Introduction}

Given a set $X$ of $n$ balls and an unknown coloring of $X$ with a fixed set of colors, we say that a ball $x\in X$ is a \emph{majority ball} if its color class contains more than $|X|/2$ balls.
The \emph{majority problem} is to find a majority ball (or show that none exists).
In the basic model of majority problems, one is allowed to ask queries of pairs $(x,y)$ of balls in $X$ to which the answer tells whether the color of $x$ and $y$ is the same or not, which we denote by SAME and DIFF, respectively.
The answers are given by an \emph{adversary} whose goal is to force us to use as many questions as possible.
It is an easy exercise to see that if the number of colors is two, then in a non-adaptive search (all queries must be asked at once) the minimum number of queries to solve the majority problem is $n-1$, unless $n$ is odd, in which case $n-2$ queries suffice.
On the other hand, Fisher and Salzberg \cite{FS1982} proved that if we do not have any restriction on the number of colors, $\lceil 3n/2 \rceil - 2$ queries are necessary and sufficient to solve the majority problem adaptively (the answer to a query is known before asking the next one).
If the number of colors is two, then Saks and Werman \cite{SW1991} proved that the minimum number of queries needed in an adaptive search is $n-b(n)$, where $b(n)$ is the number of $1$'s in the binary form of $n$ (we note that there are simpler proofs of this result, see \cite{A2004,HayKutMel02,W2002}).
There are several other generalizations of the problem, which include more colors \cite{adm,cgp,GerKatPal13,GerLenViz17}, larger queries \cite{bor,cgp,dkw,eh,GerKesPal17,GerLenViz17,GerViz18}, non-adaptive \cite{A2004,cgmy,GerKatPal13} and weighted versions \cite{GerKatPal13}. 

In the present paper we study the adaptive majority problem for two colors when we restrict the set of pairs that can be queried to the edges of some graph $G$ on $n$ vertices.
The original majority problem, where we can ask any pair, corresponds to $G=K_n$.
To distinguish between the version when we are restricted to the edges of a graph, and the original, unrestricted version, we call the colored objects \emph{vertices} and \emph{balls}, respectively.
To the best of our knowledge, the only similar result is \cite{dp}, where it was shown that if the size of the two color classes differs only by a constant, then $\Omega(n)$ queries might be needed on any graph even for a randomized algorithm to solve the majority problem, but if the sizes differ by $\Omega(n)$, then randomized algorithms can do better for several graph classes.
In this paper, however, we only deal with the worst-case performance of deterministic algorithms, which allows us to obtain better bounds.

Notice that it is possible to solve the majority problem (with any number of queries) if and only if $G$ is connected when $n$ is even, and if and only if $G$ has at most two components when $n$ is odd.
For any such graph, denote the minimum number of queries needed to solve the majority problem in the worst case by $m(G)$.
Obviously we have $n-b(n) =m(K_n)\le m(G)\le n-1$ (moreover, $m(G)\le n-2$ when $n$ is odd).
Our main results are the following.

\begin{theorem}\label{thmtree}
For every tree $T$ on an even number $n$ of vertices $m(T)=n-1$ and\\
for every tree $T$ on an odd number $n$ of vertices $m(T)\ge n-65$.
\end{theorem}

The constant $65$ is probably far from optimal. We will see trees $T$ on $n$ vertices with $m(T)=n-3$, but it is possible that $m(T)\ge n-3$ holds for every tree. We have a better lower bound, $n-6$, for paths. 

We also study the least number of edges a graph must have if we can solve the majority problem as fast as in the unrestricted case, i.e., when $m(G)=n-b(n)$.

\begin{theorem}\label{minedge} For every $n$, there is a graph $G$ with $n$ vertices and $n(1+b(n))$ edges such that $m(G)=n-b(n)$.
\end{theorem}

It would be interesting to determine whether this bound can be improved to $O(n)$, or show a superlinear lower bound.\\

The proof of Theorem \ref{thmtree} uses a \emph{weighted} version of the original (i.e., $G=K_n$ case of the) majority problem, which is defined in the next section.
We think these results are interesting on their own.

In the following, we always suppose that only two colors are used, which we call red and blue.
When both colors contain the same number of balls, then we call the coloring \emph{balanced}.

\section{Weighted majority problems}

Now we define a variant of the majority problem, where the balls are given different weights. 
More precisely, given $k$ balls with  non-negative integer weights $w_1,\dots, w_k$, a ball is a (weighted) \emph{majority ball} if the weight of its color class is more than $\sum_{i=1}^kw_i/2$.
The (weighted) \emph{majority problem} is to find a majority ball (or show that none exists). We will often identify a ball with its weight or its index and talk about a ball with weight $w_i$, or the ball $w_i$, or the ball $i$.

Note that during the running of an adaptive algorithm solving the non-weighted majority problem, at any point the information obtained so-far can be represented by a graph whose vertices are all balls, and the queries asked are edges labeled with DIFF or SAME.
Since now we study the majority problem only for two colors, we can deduce from the labels of the edges the color partition inside every component.
Denote the difference between the sizes of the color classes in each component by $w_i$.
Finishing the algorithm from a given state is equivalent to solving the majority problem with the weights $w_i$.\footnote{This is explained in more details in Section \ref{sec:graphs}.}
Similarly, in the weighted version when we ask a ball with weight $w_i$ and a ball with weight $w_j$, we can consider the answer as merging the two balls into a ball with weight $w_i+w_j$ or $|w_i-w_j|$, depending on the answer.
We will say that the new ball \emph{contains} the two previous balls.


A set of $k$ balls with given weights $w_1,\dots, w_k$ can be represented by a vector $\underline{w}=(w_1,\dots,w_k)$.
We denote the number of queries needed to solve the weighted majority problem in the worst case by $m(\underline{w})$.
So, with this notation, the result of Saks and Werman for the non-weighted problem can be written as $m(1,\ldots,1)=k-b(k)$.
Note that $m(\underline{w}) \le k-1$ and if $\sum^k_{i=1} w_i$ is odd, then $m(\underline{w}) \le k-2$ (if $k\ge 2$).




The weighted problem was first studied in \cite{GerKatPal13}, where the following proposition (which also implies the result of Saks and Werman) was proved, generalizing a result of \cite{HayKutMel02} (which built on \cite{RivVui76}) about the non-weighted variant.
Let $\mu(k)$ denote the largest $l$ such that $2^l$ divides $k$ (and define $\mu(0)=\infty$).
For $\underline{w}=(w_1,\dots,w_k)$ denote by $p$ the number of balanced colorings and by $p_i$ the number of (non-balanced) colorings such that $w_i$ is in the majority class.\footnote{Beware that in \cite{GerKatPal13} a slightly different notation was used, where $p$ denoted the number of balanced $2$-partitions, which is half of the number of balanced colorings, and part (ii) of Proposition \ref{dectree} was not explicitly stated.} 

\begin{proposition}\label{dectree} 
\textbf{(i)} $m(\underline{w})\ge k-\mu(p)$.

\textbf{(ii)} $m(\underline{w})\ge k-1-\mu(p_i)$ for every $i\le k$.
\end{proposition}

It was also shown in \cite{GerKatPal13} that $m(\underline{w})=k-\mu(p)$ for $\mu(p)\le 2$, but not in general, e.g., for $\underline{w}=\{1,2,3,4,5,6,7\}$ we have $8$ balanced colorings, but $m(\underline{w})=5>7-\mu(8)$.

Our main results about the weighted majority problem are exact bounds for some special $\underline w$.
They are based on the following lemma.

\begin{lemma}\label{suly1lemma} Let $\underline{w}=(w_1,\dots,w_k)$ and $k> 2^n$.

\textbf{(i)} If $w_1=\dots =w_{2^n}=1$ and $\sum_{i=1}^k w_i=2^{n+1}$, then $m(\underline{w})= k-1$.

\vspace{1mm} 

\textbf{(ii)} If $w_1=\dots =w_{2^n}=1$, $\sum_{i=1}^k w_i=2^{n+1}+1$ and $k\ne 2^n+1$, then $m(\underline{w})= k-2$.
\end{lemma}


\begin{proof} For both statements we use Proposition \ref{dectree}. Let us start with \textbf{(i)}.  We are going to calculate $\mu(p)$.
First color only the balls whose index is from $\{2^n+1,\ldots,k\}$.
Denote by $B$ and $R$, respectively, the blue and red balls whose index is from $\{2^n+1,\ldots,k\}$.
Let $x:=|\sum_{i\in B} w_i-\sum_{i\in R} w_i|$. Note that $x$ is an even integer, and for a fixed $x$ there are $2\binom{2^n}{2^{n-1}-x/2}$ ways in which we can color $\{1, \ldots,2^n\}$ to make the coloring balanced (note that this value is obtained by considering both the cases $\sum_{i\in B} w_i-\sum_{i\in R} w_i\le 0$ and $\sum_{i\in B} w_i-\sum_{i\in R} w_i\ge 0$).
It is easy to see that $2\binom{2^n}{2^{n-1}-x/2}$ is divisible by $4$ except for the case $x=2^n$, when this number is exactly $2$. This means that the number of balanced colorings is $2\bmod 4$. Thus Proposition \ref{dectree} (i) gives $m(\underline{w}) \ge k-1$, and since $m(\underline{w}) \le k-1$ always holds, we have equality.

The proof of \textbf{(ii)} goes similarly. We are going to calculate $\mu(p_k)$. We define $B$ and $R$ in the same way. Without loss of generality we can assume that ball $k$ is blue and let $y=\sum_{i\in B} w_i-\sum_{i\in R} w_i$. Then the number of colorings of the first $2^n$ balls such that blue is the majority color is $\sum_{i=2^{n-1}-\lfloor y/2 \rfloor}^{2^n} \binom{2^n}{i}$. If $y<2^n$, each term here is divisible by 2, except the last one, while if $y>2^n$, each term is divisible by 2.
There are $2^{k-1-2^n}$ ways to color the balls from $\{2^n+1,\ldots,k-1\}$ out of which $y>2^n$ only once.
This means that when ball $k$ is blue, then the number of colorings (of \emph{all} the balls) where blue is the majority color is odd. Of course, the same is true when ball $k$ is red. Thus Proposition \ref{dectree} (ii) gives $m(\underline{w}) \ge k-2$, and since $m(\underline{w}) \le k-2$ holds whenever $\sum^k_{i=1} w_i$ is odd, we have equality.
\end{proof}

Note that in the above proof we do not need at all that the first few weights are $1$, we only need that they are the same and that their number is a power of two.
We also do not need in \textbf{(i)} that the sum of all the weights is exactly $2^{n+1}$, only that $x=\sum_{i\in B} w_i-\sum_{i\in R} w_i=2^n$ can happen in an odd number of ways.
We also do not need in \textbf{(ii)} that the sum of all the weights is exactly $2^{n+1}+1$, only that $-2^n< y=\sum_{i\in B} w_i-\sum_{i\in R} w_i\le 2^n$ can happen in an odd number of ways (with fixed ball $k$ always blue).
A sufficient condition for this is that $y>-2^n$ always holds, while $y\le 2^n$ holds except when all balls are blue, i.e., we need $w_k-\sum_{i=2^n+1}^{k-1} w_i>-2^n$ and $\sum_{i=2^n+1}^k w_i-w_j\le 2^{n}+w_j$ for any $2^n<j<k$.
To summarize, this proves the following.

\begin{lemma}\label{suly1formalemma} Let $\underline{w}=(w_1,\dots,w_k)$ and $k> 2^n+1$.

\textbf{(i)} If $w_1=\dots =w_{2^n}$ and there are an odd number of partitions $R\cup B=\{2^n+1,\ldots, k\}$ such that $\sum_{i\in B} w_i-\sum_{i\in R} w_i=w_12^n$, then $m(\underline{w})= k-1$.

\vspace{1mm} 

\textbf{(ii)} If $w_1=\dots =w_{2^n}$ and $\sum_{i=1}^k w_i\le w_12^{n+1}+2w_j$ for any $2^n<j\le k$, with the inequality being strict for $j=k$, then $m(\underline{w})\ge k-2$.
\end{lemma}

These imply, for example, that $m(3,3,7,8,9)=4$ and $m(3,3,5,5,5)\ge 3$.

For simplicity, we state the later consequences only for Lemma \ref{suly1lemma}, but similar generalizations of Lemma \ref{suly1formalemma} also hold.

\begin{corollary}\label{suly1cor} 
\textbf{(i)} If $w_1=\dots =w_{2^n+2s}=1$ and $\sum_{i=1}^k w_i=2^{n+1}+2s$, then $m(\underline{w})\ge k-1-s$.

\vspace{1mm} 

\textbf{(ii)} If $w_1=\dots =w_{2^n+2s}=1$, $\sum_{i=1}^k w_i=2^{n+1}+2s+1$ and $w_k\ne 2^n+1$, then $m(\underline{w})\ge k-2-s$.
\end{corollary}
\begin{proof}
We only prove \textbf{(i)} - the proof of \textbf{(ii)} goes the same way.
First, we prove the weaker statement that $m(\underline{w})\ge k-1-2s$.
We reveal $2s$ balls of weight 1 such that half of them are red and half of them are blue, and then apply Lemma \ref{suly1lemma} for the remaining balls.

For the bound $m(\underline{w})\ge k-1-s$ we need one more trick.
We run the adversarial algorithm that gives the lower bound in Lemma \ref{suly1lemma}, until the first ball of weight 1 is queried.
When this happens, then we reveal that it is red, and we reveal another ball of weight 1 that it is blue.
We do this $s$ times, and after that proceed according to the adversarial algorithm.
\end{proof}

Call a vector $\underline{w}=(w_1,\dots,w_k)$ \emph{hard} if $m(\underline{w})=k-1$ and $\sum_{i=1}^n w_i$ is even, or $\sum_{i=1}^n w_i$ is odd and $m(\underline{w})=k-2$.
Thus Lemma \ref{suly1lemma} states that the vectors satisfying its conditions are hard.


\begin{observation}\label{obs}
Let $\underline{w}=(w_1,w_2,\dots,w_k)$ be a hard vector with $w_1=w_2$ and let $\underline{w}'=(2w_1,w_3,w_4,\dots,w_k)$. Then $\underline{w}'$ is 
also hard.
\end{observation}
\begin{proof}
 If $k=2$, then $\underline{w'}$ is hard by definition.
 Let us assume that $k\ge 3$ and $\underline{w}'$ is not hard. We will show that $\underline{w}$ is not hard either. We ask $w_1$ and $w_2$ in the first query. If the answer is DIFF, we can obviously finish with $k-3$ further queries if $\sum_{i=1}^k w_i$ is even and $k-4$ further queries if $\sum_{i=1}^k w_i$ is odd, thus  $\underline{w}$ cannot be hard. If the answer is SAME, we apply our algorithm for $\underline{w}'$ to reach the same conclusion using that $\underline{w}'$ is not hard.
\end{proof}

\begin{question}
    Does the reverse direction also hold in the above observation?
\end{question}

Also, the respective statement might hold when $m(\underline{w})$ is smaller, but we do not know of other, generally applicable sufficient conditions.

Combining Observation \ref{obs} with Lemma \ref{suly1lemma}, we obtain the following statement.

\begin{lemma}\label{suly2lemma}  
If $w_1,\dots,w_j$ are each powers of two and 

\vspace{2mm}

\textbf{(i)} $\sum_{i=1}^{j} w_i=2^{n}$ and $\sum_{i=1}^k w_i=2^{n+1}$, then $\underline{w}$ is hard, i.e.,  $m(\underline{w})= k-1$.

\vspace{1mm} 

\textbf{(ii)} $\sum_{i=1}^{j} w_i=2^{n}$, $\sum_{i=1}^k w_i=2^{n+1}+1$ and $k>2^n+1$, then $\underline{w}$ is hard, i.e., $m(\underline{w})= k-2$.
\end{lemma}

Note that \textbf{(i)} of Lemma \ref{suly1lemma} states that if the sum of the weights in $\underline{w}$ is a power of two, and at least half of that weight is given by balls of weight 1, then $\underline{w}$ is hard. Lemma \ref{suly2lemma} shows that weight 1 can be replaced by any weights that are powers of two, and  \textbf{(ii)} of Lemma \ref{suly1lemma} can be similarly improved.
If Observation \ref{obs} held for non-hard vectors as well, then we could obtain an improvement of Lemma \ref{suly2lemma}, similar to Corollary \ref{suly1cor}.
Instead, we state the following weaker statement.

\begin{corollary}\label{suly2cor} 
If $w_1,\dots,w_j$ are each powers of two and 

\vspace{2mm}

$\sum_{i=1}^{j} w_i=2^{n}$, $w_{j+1}=1$, $\sum_{i=1}^k w_i=2^{n+1}+3$ and $k>2^n+2$, then $m(\underline{w})\ge k-3$.
\end{corollary}
\begin{proof}
If $w_i=1$ for every $i>j$, then we know that $m(\underline{w})=k-2$.
Otherwise, before the start of the algorithm, we can reveal that ball $j+1$ and the heaviest ball with index $>j+1$ have different colors, reducing the problem to \textbf{(ii)} of Lemma \ref{suly2lemma}.
\end{proof}

Combining Lemma \ref{suly1formalemma} and Corollary \ref{suly2cor} we obtain the following.

\begin{proposition}\label{O1G}
If $1\le w_1,\dots,w_j\le 2$ and 

\vspace{2mm}

$\sum_{i=1}^{j} w_i=2^{n}$,  $\sum_{i=1}^k w_i=2^{n+1}+3$ and $k>2^n+2$, then $m(\underline{w})\ge k-3$.
\end{proposition}
\begin{proof}
If there is some $i>j$ such that $w_i=1$, we are done using Corollary \ref{suly2cor}.
Otherwise, for all $j<i\le k$ we have $w_i\ge 2$, and we can apply Lemma \ref{suly1formalemma}.
\end{proof}










The next subsection, contrary to its title, is not so relevant to our proof, but it helps to understand better what can happen before the final steps of an optimal algorithm that solves the majority problem.

\subsection{Relevant balls}
Given $\underline{w}=(w_1,\dots,w_k)$, we say that $w_i$ is \emph{relevant}\footnote{In \cite{GerKatPal13} the property that every ball is relevant was called \emph{non-slavery}.} if there is a coloring of the other balls such that the color of $w_i$ changes what the majority color is, or whether a majority color exists. In other words, there is a coloring of the other balls, such that either $w_i$ is red means red is majority and $w_i$ is blue means blue is majority, or one color of $w_i$ means there is no majority, the other color means there is majority.
In this subsection we prove some simple facts about relevant balls.
We start with some simple observations.


\begin{proposition}\label{rele1} \textbf{(i)} If $m(\underline{w})=0$, then either there is no relevant ball and the answer is that there is no majority, or there is one relevant ball and that is the majority ball.

\textbf{(ii)} If we obtain $\underline{w}'$ from $\underline{w}$ by any answer to a query $(a,b)$ such that $b$ is non-relevant, then $m(\underline{w}')=m(\underline{w})$.

\textbf{(iii)} If we increase the weight of a relevant ball, it remains relevant. In other words, if we obtain $\underline{w}'$ from $\underline{w}$ by replacing one ball $w_i$ with $w_i'$ such that $w_i'>w_i$, and $w_i$ is relevant in $\underline{w}$, then $w_i'$ is relevant in $\underline{w}'$.

\textbf{(iv)} For any $\underline{w}$ there is a threshold $t>0$ such that $w_i$ is relevant if and only if $w_i>t$.

\textbf{(v)} If a ball $x$ is relevant before a query $Q$ not containing $x$, there is at least one answer to $Q$ such that $x$ is still relevant afterwards. If a ball $x$ is relevant before a query $(x,y)$, then after the answer SAME the resulting ball with weight $x+y$ is relevant.

\textbf{(vi)} For any query $(a,b)$ there is an answer such that the number of relevant balls decreases by at most two.
\end{proposition}

\begin{proof} To prove \textbf{(i)}, observe that if we color all the balls blue, there is a majority, unless all the balls have zero weight, in which case there is no relevant ball. Thus if there is no majority in any coloring of $\underline{w}$, then there cannot be relevant balls. If there is a majority ball $a$, then it has to be the only relevant ball. Indeed, if $b$ is relevant, then changing the color of $b$ must change the answer, unless $b$ was the answer.

To prove \textbf{(ii)}, let $c$ be the ball that the answer to the query $(a,b)$ gives, thus it has weight either $a+b$ or $|a-b|$. Any algorithm that gives a solution for $\underline{w}$ gives a solution for $\underline{w}'$, where asking a ball that contains $a$ is replaced by the query that contains $c$, and ignoring queries that involve $b$. A similar argument shows that a solution for $\underline{w}'$ gives a solution for $\underline{w}$.
Therefore, $m(\underline{w}')=m(\underline{w})$.

To prove \textbf{(iii)}, assume for a contradiction that $w_i'$ is not relevant and take a coloring of the other balls that shows this.
But then taking the same coloring for the other balls also shows that $w_i$ is not relevant, a contradiction.

To prove \textbf{(iv)}, observe first that it is equivalent to the statement that a non-relevant ball cannot have larger or equal  weight than a relevant ball. Indeed, this is obviously implied by \textbf{(iv)}, and if this holds, than $t$ can be chosen as the largest weight of a non-relevant ball. Now assume $a$ is relevant and $w(b)\ge w(a)$. Then consider the coloring that shows $a$ is relevant, and exchange the color of $b$ and the color of $a$. This coloring clearly shows $b$ must be also relevant. 

To prove \textbf{(v)}, assume $x$ is not in the query and consider a coloring of the other balls such that the color of $x$ decides the majority. In that coloring the balls in the query have different or same color; answer accordingly. Then the same coloring shows $x$ is still relevant. 

Assume now $x$ is in the query $(x,y)$ and the answer is SAME.
Consider a coloring of the other balls such that the color of $x$ decides the majority.
Taking the same coloring, the new ball $x+y$ will decide majority.

To prove \textbf{(vi)}, consider the relevant ball $x$ not in the query with the smallest weight. By \textbf{(v)} there is an answer such that $x$ remains relevant.
As other relevant balls have larger weight, they also remain relevant, except for $a$ and $b$ (whose total weight can go below the weight of $x$).
\end{proof}

We remark that \textbf{(vi)} of the above proposition gives a new proof of a proposition from \cite{GerKatPal13}, which states that if all the $n$ balls are relevant, we need at least $\lfloor n/2\rfloor$ queries. 

\begin{proposition}
Before the last query of an optimal algorithm, there are either two relevant balls, they are of equal weight and there are no other balls with non-zero weight, or there are three relevant balls, and any query that compares two of them finishes the algorithm.
\end{proposition}

\begin{proof} 


Proposition \ref{rele1} implies that there are at most three relevant balls. Observe that if there are two relevant balls $a$ and $b$ in $\underline{w}$, they must have the same weight. Indeed, if $w(a)>w(b)$ and the total weight $m$ of the other balls is smaller than $w(a)-w(b)$, then $a$ is the only relevant ball. If $m\ge w(a)-w(b)$, then color $b$ red, $a$ blue and go through the other balls in increasing order of their weight, without the last ball $c$. We give each of them the color which has the smaller weight at that point. The first ball gets the color red, but as $m\ge w(a)-w(b)$, at one point the total weight of red balls becomes at least the total weight of blue balls. From that point, the difference between the classes is at most the weight of the current ball, which is at most $w(c)$. This coloring shows $c$ is relevant.

It is left to show that if there are exactly three relevant balls, $a$, $b$ and $c$, querying any two of them (say $a$ and $b$) finishes the algorithm. Let $m$ be the sum of the weights of the other balls. If $m=0$, we are done unless both $a+b$ and $|a-b|$ are equal to $c$, which means $b=0$, but a ball with zero weight cannot be relevant. Thus we can assume $m>0$. We have $a\le b+c+m$, otherwise we are done (which contradicts our assumption that we are before the last query). But we also have $a+m\le b+c$, because the other balls are not relevant. Moreover, if $a+m=b+c$, then again, some of the other balls would be relevant, thus we have $a+m<b+c$. This implies $c> a-b+m$, thus we are done if the answer is DIFF, as $c$ is a majority ball. We also have $a+b+m\ge c$, otherwise we are done without the last query, and it implies $a+b\ge c+m$, moreover $a+b>c+m$, otherwise some of the remaining balls are relevant. Thus we are done if the answer is SAME.
\end{proof}

\section{Graphs}\label{sec:graphs}

Let us start this section with describing in detail how the weighted majority problems are connected to the majority problem on graphs. Consider an algorithm solving the majority problem on a graph $G$. Let $G^i$ be the subgraph of $G$ formed by the first $i$ queries. We call the vertex set of a connected component of $G^i$ a \emph{$q$-component}. Observe that knowing the answer to the first $i$ queries, for every $q$-component $U$ we know a partition of $U$ into a blue subset $U_1$ and a red subset $U_2$. Let the weight of $U$ be $w^i(U)=\big||U_1|-|U_2|\big|$. If the $i+1^{\mathrm{st}}$ query is $(u,u')$ with $u\in U$ and $u'\in U'$, where $U$ and $U'$ are $q$-components in $G^i$, then $U$ and $U'$ are merged into a $q$-component with vertex set $U\cup U'$ in $G^{i+1}$. Moreover, its weight $w^{i+1}(U\cup U')$ is either $w^i(U)+w^i(U')$, or $|w^i(U)-w^i(U')|$, depending on the answer to the $i+1^{\mathrm{st}}$ query $(u,u')$. For other $q$-components $U''$ we have $w^{i+1}(U'')=w^i(U'')$. Hence an algorithm that finishes solving the majority problem on $G$ after the $k^{\mathrm{th}}$ query also solves the weighted majority problem for the vector having the weights of the $q$-components of $G^k$ as coordinates. However, this does not work in the other direction, as we do not have the restriction of the graph structure in the weighted problem. Thus we can only prove upper bounds for $m(G)$ this way. 

We will omit $i$ and simply talk about $w(U)$ instead of $w^i(U)$ because $i$ will be always clear from the context. 
For a ball $u\in U$, let $w(u):=w(U)$. If a $q$-component $X$ has weight zero, we say that $X$ is \emph{balanced}.
Similarly to vectors, we say that a graph $G$ on $n$ vertices is \emph{hard} if $m(G)=n-1$ for even $n$ and $m(G)=n-2$ for odd $n$.

\begin{proposition}\label{eventrees} Every tree $T$ on an even number $n$ of vertices is hard, i.e., $m(T)=n-1$.
\end{proposition}
\begin{proof}
 We show more: the adversary can pick in advance a coloring $c$ of the vertices such that no matter what edge is missing from the queries, we cannot find out if there is a majority or not. All this coloring needs to satisfy is that we have $n/2$ blue balls, and if we remove any edge from $T$, both the resulting subtrees are unbalanced, i.e., the number of blue and red balls is not the same in them. Indeed, if an edge of $T$ is not asked, then the adversary can either claim that the real coloring of the vertices is $c$ and thus no majority vertex exists, or the coloring coincides with $c$ on one component, but is exactly the flipped version of $c$ on the other component and thus there is majority.
 
 Equivalently, we want to find a balanced 2-coloring such that each edge of $T$ cuts it into two non-balanced parts. We start with an arbitrary balanced coloring of $T$. If an edge connecting a red ball $u$ and a blue ball $v$ cuts $T$ into two balanced parts, we can simply change the color of $u$ to blue and the color of $v$ to red. Observe that any other edge $e$ cuts $T$ into two parts such that $u$ and $v$ belong to the same part, hence it does not change whether $e$ cuts $T$ into balanced parts.
 
 Let us assume now that $u$ and $v$ are both red, and the edge $uv$ cuts $T$ into two balanced parts $A$ and $A'$. Then we change the color of every ball in $A$. We claim that for any edge $e$ that cuts $T$ into parts $B$ and $B'$, it does not change whether $B$ and $B'$ are balanced. Indeed, either $B$ is completely inside $A'$, in which case no color inside $B$ is changed, or $B$ contains $A$, in which case some colors have changed, but the number of blue balls turning red is the same as the number of red balls turning blue, as $A$ is balanced. As $B'$ is balanced if and only if $B$ is balanced, $B'$ is also unaffected. Now $u$ and $v$ have different colors, so we can again exchange their color.
 
Hence we obtained that we can decrease the number of edges that cut $T$ into balanced parts. It is easy to see that the coloring remains balanced. After applying this operation finitely many times we obtain the desired coloring, finishing the proof.
\end{proof}

Surprisingly, it is much harder to give a lower bound for trees on an odd number of vertices.
For paths, for example, we have $m(P_n)=n-b(n)$ for all odd $n\le 13$, while $m(P_{15})=12=n-b(n)+1=n-3$.
(This we have verified with a computer program.)
We conjecture that $n-3$ might be a lower bound for all trees, but we can only prove the weaker bound $n-65$.


To prove the lower bound of $n-65$ for odd $n$, we start with a lemma that gives another proof for Proposition \ref{eventrees}.
First, we introduce a notation.
In a graph $G$, for a subset of its vertices $X\subset V$ we denote by $\delta(X)$ the \emph{parity} of the number of edges between $X$ and $V\setminus X$.
If $G$ is a tree and $X$ is a connected subset of vertices, then $\delta(X)$ equals the parity of the number of components of $V\setminus X$. Recall that the weight of a $q$-component $X$, denoted by $w(X)$ is the difference between the number of the blue balls and the number of the red balls in it.

\begin{lemma}\label{treelemma}
    We can answer to any sequence of queries in any graph $G$ such that for any $q$-component $X\subsetneq V$ we have $0\le w(X)\le 2$, and\\
\textbf{(i)} if $|X|$ is odd, $w(X)=1$, \\
\textbf{(ii)} if $|X|$ is even, $w(X)=2\delta(X)$.
\end{lemma}


Note that if $T$ is a tree on an even number of vertices, then $0\le w(X)\le 2$ implies that the game goes on until we have at most one non-balanced $q$-component.
Assume at that point there would also be some balanced $q$-components. Observe that there is a tree-structure on the $q$-components of a tree. Then at least one of the balanced $q$-components would be a leaf-component, but that contradicts condition \textbf{(ii)}.
Thus there can be only one component, which implies Proposition \ref{eventrees}.

For trees on an odd number of vertices, a similar argument cannot work, as for example in $P_n$, it can happen that the first two vertices form a $q$-component of weight $2$, followed by $(n-3)/2$ pairs of vertices that each form a balanced $q$-component, and the last vertex is a $q$-component of weight $1$.
In this case we have solved the majority problem with only $(n-1)/2$ queries.
In fact, according to the conditions of Lemma \ref{treelemma}, our algorithm would answer exactly so that it would produce such weights for the $q$-components.
For paths, there is no way to keep the weight function bounded without allowing an arbitrarily number of adjacent balanced $q$-components; but if this happened, then we could merge all the $q$-components to their left, and all the $q$-components to their right, so that only two non-balanced $q$-components remain - after this we are done if $n$ is odd, saving an arbitrarily large number of queries.
This is why the proof will be more complicated for trees on an odd number of vertices; we will need to use our results about weighted balls.

\begin{proof}[Proof of Lemma \ref{treelemma}]

    
    
    
    
    
    Initially the conditions are satisfied.
    Suppose that the query is between two $q$-components, $X$ and $Y$.
    
    If $|X|+|Y|$ is odd, then exactly one of $w(X)$ and $w(Y)$ equals $1$, while the other equals $0$ or $2$, so we can achieve $w(X\cup Y)=1$ to satisfy condition \textbf{(i)}.
    
    If $|X|$ and $|Y|$ are both odd, then we can choose the weight of $X\cup Y$ to be 0 or 2; one of those is equal to $2\delta(X)$.
    
    If $|X|$ and $|Y|$ are both even, then since $\delta(X\cup Y)=\delta(X)+\delta(Y)-2|E(X,Y)|=\delta(X)+\delta(Y) \bmod 2$, we need $w(X\cup Y)=w(X)+w(Y) \bmod 4$. Observe that $w(X)+w(Y)$ is 0, 2 or 4. Thus we can answer so that $w(X\cup Y)$ becomes 0, 2 or 0, respectively, to satisfy condition \textbf{(ii)}.
\end{proof}

For the lower bound of $n-65$ for trees on an odd number of vertices, we need another theorem. Before that, we prove a simpler result that contains an important ingredient of the proof, and is of independent interest.

\begin{theorem}\label{lefogo1} Let $n=2^k+l$, where $l<2^k$. If $G$ has a set $U$ of vertices such that $|U|\le 2^{k-2}$ and the components of $G\setminus U$ are single vertices (i.e., every edge is incident to a vertex in $U$), then $G$ is hard, i.e., $m(G)=n-1$ if $n$ is even and $m(G)=n-2$ if $n$ is odd.
\end{theorem}

\begin{proof} Denoting by $w(X)$ the weight of a $q$-component $X$, we initially have $\sum_X w(X)=n$.
The adversary will maintain in the first part of the algorithm that $w(X)\ne 0$ for every $q$-component $X$.

Let us now describe a strategy of the adversary for the first part of the algorithm.
Whenever for some $v\in G\setminus U$ we ask the first query containing $v$, if the other vertex in the query $u\in U$ is such that $w(u)=p\ge 2$, the answer is such that the weight of the new $q$-component is $p-1$, thus $\sum_X w(X)$ decreases by $2$.
In every other case the answer is such that the weights are added up, i.e., $\sum_X w(X)$ remains the same.

Introduce the potential function\footnote{Instead of the potential function argument, the proof of Theorem \ref{lefogo1} could also be finished in the same way as the proof of Theorem \ref{oddpath} or \ref{lefogo2}.} $\Psi=\sum_X w(X) + |\{X\mid X\cap U\ne \emptyset,\, w(X)=1\}|$.
The adversary's strategy is such that every time we ask the first query containing a $v\in G\setminus U$, the function $\Psi$ decreases by at least $1$.
Since initially $\Psi=n+|U|$, after $|U|+l$ queries involving some vertex of $V\setminus U$, we would have $2^k\ge \Psi\ge \sum_X w(X)$.
But the adversary stops executing this algorithm the moment we have $\sum_X w(X)=2^k$ or $\sum_X w(X)=2^k+1$; this surely happens, as $\sum_X w(X)$ can only decrease by $2$.

Let us consider the vertices from $G\setminus U$ that were merged into some $q$-components (i.e. those that appeared in queries). Let $x$ denote the number of those where the total weight did not decrease when they first appeared in a query, and $y$ denote the number of those where the total weight decreased when they first appeared in a query. Then we have $x\le y+|U|$. Indeed, consider a $q$-component containing a vertex $u\in U$, and observe that whenever the weight of this component increased by merging it with a vertex from $G\setminus U$, the next time its weight decreased. 

This implies that at the point where the adversary stops executing the algorithm, the number of vertices in $G\setminus U$ that have not appeared in any query is at least $n-|U|-(|U|+l)\ge 2^{k-1}$.
Now we can apply Lemma \ref{suly1lemma} to the current $q$-components as weighted balls. Indeed, we have at least $2^{k-1}$ $q$-components of weight 1, and the total weight is $2^k$ or $2^k+1$. By Lemma \ref{suly1lemma} the number of queries needed is the number of components minus 1 or minus 2, depending on the parity. 
Hence even if we could compare any two $q$-components from now, we still could not solve the majority problem with less queries.
\end{proof}

With a similar method, we can obtain the following lower bound for odd paths.

\begin{theorem}\label{oddpath}
$m(P_n)\ge n-6$.\\
Moreover, $m(P_n)\ge n-5$ unless $n+1$ or $n+3$ is a power of two.
\end{theorem}
\begin{proof}
We have already seen that this holds if $n$ is even, so it is enough to prove the theorem for $n$ odd.
First we prove the weaker claim $m(P_n)\ge n-10$.
The statement holds for $n<1000$ as $m(P_n)\ge n-b(n)$.
Let $U$ include every $9^{\mathrm{th}}$ vertex of $P_n$, starting with the first, and also the last vertex of $P_n$, so $\lceil \frac n9 \rceil\le |U|\le \lceil \frac n9 \rceil+1$, and $P_n\setminus U$ consists of paths on $8$ vertices (and possibly one shorter path at the end).
We answer each query such that for any $q$-component $X$ if $X\cap U=\emptyset$, then $w(X)\le 1$, while if $X\cap U\ne \emptyset$, then $1\le w(X)\le 2$.
In each step the total weight decreases by $0$ or $2$, so after a while it becomes $2^k+1$ for $k=\lfloor \log n\rfloor$.
When this happens, we apply 
Lemma \ref{suly1lemma} to the current $q$-components as weighted balls.
Indeed, $\sum_{X: X\cap U\ne \emptyset} w(X)\le 2|U|=2\lceil \frac n9 \rceil+2\le \frac n4 \le  2^{k-1}$ if $n>1000$, so $\sum_{X: X\cap U= \emptyset} 1 > 2^{k-1}$.
By Lemma \ref{suly1lemma}, the number of queries needed to finish is at least the number of components minus 2. Equivalently, Lemma \ref{suly1lemma} states that we need to connect the weighted balls until at most two components remain, thus, we need to connect all of $U$ into at most two components. This means querying all the edges between any two vertices of $U$ that are in the same component at the end. That means the edges we have not queried are all on a path of length $9$ (between two vertices from $U$).
This proves $m(P_n)\ge n-10$.

But we can do even better, because out of the $9$ edges of the path at least $4$ must be queried if the path contains no non-balanced components.
This proves $m(P_n)\ge n-6$ if $n$ is large enough, but now we have to be more careful with the calculations.
Because of this, we also change how we select $U$; instead of starting with the first vertex, we start with the second vertex of the path, then take every $9^{\mathrm{th}}$ vertex, and finally the last but one vertex.
We can afford to skip the endvertices, as a single vertex anyhow cannot form a balanced component, we can only compare it to its adjacent vertex from $U$.
This gives $|U|=\lfloor \frac{n+14}9\rfloor$, and $\frac{n+14}9\le\frac n8$ if $n\ge 112$, while for $n<127$ the lower bound $n-b(n)\ge n-6$ holds.

The proof of the moreover part is similar, except that after we start with the second vertex, we take every $8^{\mathrm{th}}$ vertex, and finally the last but one vertex.
This way only $4$ edges can remain unqueried between two different components.
This gives $|U|=\lfloor \frac{n+12}8\rfloor$, and this is less than $2^{\lfloor\log_2 n\rfloor -2}$ unless $n+1$ or $n+3$ is a power of two.
\end{proof}

It is an interesting question where the truth is between $n-3$ and $n-6$ for $P_n$ for odd $n$.
Our only (computer verified) case is $m(P_{15})=n-3$.\\

Now we present the lower bound for general trees.
The proof of Theorem \ref{lefogo2} is based on a similar idea as that of Theorem \ref{lefogo1}, but also combines ideas from Theorem \ref{oddpath} and uses Proposition \ref{O1G}.
We also need the following version of the folklore generalization of the concept of centroid for trees, known as \emph{centroid decomposition}.

\begin{proposition}\label{centroid}
    In every tree on $n$ vertices, for every integer $p$, there is a subset $U$ of at most $2n/p$ vertices such that every component of $G\setminus U$ has at most $p$ edges (including the edges from the components to $U$).
\end{proposition}

\begin{theorem}\label{lefogo2}
If $G$ is a tree on $n$ vertices, then $m(G)\ge n-65$.
\end{theorem}

\begin{proof}
Let $n=2^k+l$, where $l<2^k$ is odd. Observe that the statement is trivial if $n\le 65$, thus we can assume $k\ge 6$.
Apply Proposition \ref{centroid} with $p=32$ to obtain a set $U$ of vertices such that $|U|\le 2^{k-3}-1$ and each component $T$ of $G\setminus U$ has at most $p$ edges.
(We write $p$ instead of $32$ throughout the proof.)

    We proceed as in the proof of Theorem \ref{lefogo1}.
    We denote by $w(X)$ the weight of a $q$-component $X$, and for a vertex $u$ of $G$, $w(u)$ denotes the weight of the $q$-component containing $u$. We initially have $\sum_X w(X)=n$.
    The adversary will maintain in the first part of the algorithm that $w(X)\ne 0$ for every $q$-component $X$ that intersects $U$.
    
    We split each component $T$ to a \emph{connecting} part $T'$ and some \emph{hanging} parts $T_1,T_2,\ldots$ where any of these can be empty, as follows.
    If $v\in T$ separates some vertices of $U$ from each other, then it goes to $T'$.
    Each connected component of $T\setminus T'$ forms a different $T_i$.
    Notice that each hanging part $T_i$ is a subtree of $T$, thus it has a unique vertex $r(T_i)$ that separates $T_i\setminus \{r(T_i)\}$ from $T\setminus T_i$; we call $r(T_i)$ the \emph{root} of $T_i$.
    
    We answer queries inside $T_i$ according to Lemma \ref{treelemma} (applied only to $T_i$), while if the query $X\cap T'\ne\emptyset$, we answer such that $w(X)\le 2$ (which is similar to Theorem \ref{oddpath}).
    This way the weight of any $X\subset G\setminus U$ will be at most $2$.
    The crucial property is that the balanced $q$-components of $T$ will always separate either two $U$ vertices, or some positive weight part of a $T_i$ from a $U$ vertex.
    This way they are ``in the way'' to compare these parts with the rest of the graph, so they cannot be simply ignored.
    The strategy of the adversary will be to make sure that the game cannot end while there are many unbalanced $q$-components.
    After there are only few unbalanced $q$-components the game might end, but in this case the graph could be made into a single $q$-component by adding $O(p)$ further edges to it.
    This shows that at most these many queries can be saved.
    
    Also, in case we merge all of some $T_i$ into one $q$-component, the adversary would like to avoid $w(T_i)=0$.
    This cannot happen if $T_i$ has an odd number of vertices; if $T_i$ has an even number of vertices, the adversary adds an (imaginary) extra degree one vertex $r'(T_i)$ to $T_i$ that is adjacent only to $r(T_i)$, to obtain $T_i^*$, and applies Lemma \ref{treelemma} to $T_i^*$ instead of $T_i$.
    Since $r'(T_i)$ is never compared with anything, merging all of $T_i$ into a $q$-component cannot give $w(T_i)=0$, because $T'=T_i^*\setminus T_i$ has only one component, $\{r'(T_i)\}$. 
    Therefore, in case the whole tree $T_i$ is merged, we get $w(T_i)=2$.
    
    Whenever we compare some $Y\subset G\setminus U$ with an $X$ intersecting $U$ such that $w(X)\ge 3$, the adversary answers such that the weight of the new $q$-component is $w(X)-w(Y)$, thus $\sum_X w(X)$ decreases by $2w(Y)\le 4$.
    In every other case the adversary answers so that the weights are added up, i.e., $\sum_X w(X)$ remains the same.
    This way the weight of a $q$-component can never exceed $4$, unless we merge two $q$-components that both intersect $U$.
    Because of this, we can conclude that $\sum_{X: X\cap U\ne \emptyset} w(X)\le 4|U|\le n/4< 2^{k-1}$.

  The adversary stops executing this algorithm the moment we have $\sum_X w(X)=2^k+1$ or $2^k+3$; this surely happens, as $\sum_X w(X)$ is odd and can decrease by at most $4$.
  As we have seen in the earlier proofs, if $\sum_X w(X)=2^k+1$, then we will have two non-balanced $q$-components when the algorithm is done.
  If $\sum_X w(X)=2^k+3$, then we can apply Proposition \ref{O1G}, whose conditions are shaped to work here, to conclude that we will have at most three non-balanced $q$-components when the algorithm is done.
     
 Moreover, these few remaining non-balanced $q$-components need to cover $U$, as the weights of sets intersecting $U$ stays positive throughout the algorithm.
 If at the end we have at most $\ell$ components, then adding $\ell-1$ original tree component $T$'s, we can make the $q$-graph connected. As every tree has at most $p$ vertices, and in our case $\ell\le 3$, adding $2p$ edges can make the $q$-graph connected.
    
To summarize, instead of asking all $n-1$ edges, we might save $2p=64$.
\end{proof}

\begin{rem}
We could get a better constant by considering the number of yet unqueried edges we need to add to connect the remaining non-balanced $q$-components (as in the end of the proof of Theorem \ref{oddpath}).
Here we will not go to details, as our bound is probably anyhow far from being optimal, but this would give something like $n-33$ as a lower bound. 
\end{rem}


\subsection{Non-deterministic complexity of odd trees}

We have seen that it is much harder to prove the lower bound $m(T) \ge n-65$ for trees of odd
order $n$ than the lower bound $m(T) \ge n-1$ for trees of even order $n$. Somewhat even more 
surprisingly, there is a significant difference between the so-called non-deterministic 
complexities of trees of even and odd order. The non-deterministic complexity $m_{nd} (G)$ of a 
graph $G$ is defined as the minimum number of queries needed to find a majority vertex in the worst case, provided we know the color of each vertex beforehand from an unreliable source and we just have to verify (some of) this information. Let us observe that in the proof of 
Proposition \ref{eventrees} we actually showed  $m_{nd} (T) = n-1$ for any tree $T$ of even 
order $n$. 

\begin{proposition} \label{nondet} 
Let $P$ be a path of order $n$, such that $n$ is odd. Then $m_{nd}(P)=n- \Theta (\sqrt{n})$.
\end{proposition}

\begin{proof} 
Let us denote the $i^{\mathrm{th}}$ vertex of $P$ by $x_i$ for $i=1,2,\ldots ,n$.
For the lower bound let us suppose that $n=k^2+1$ for some even $k$ (this is possible, since we are only interested in the order of magnitude of $n-m_{nd}(P)$) and let us call the batch of vertices $x_{(i-1)k+1}, x_{(i-1)k+2}, \ldots , x_{ik}$ Batch $i$ for $i=1,2,\ldots ,k$. Now let us color the vertices of Batch $i$ red if 
and only if $i$ is odd and let $x_{k^2+1}$ be blue (just like the vertices of Batch $k$, since $k$ is even). We claim that in order to find a majority vertex, one needs
at least $n-2k-1$ comparisons (that is, one has to verify the result of this many comparisons). Assume to the 
contrary that fewer comparisons suffice.
Then the number of edges not asked as queries is $p\ge 2k$, hence the number of
$q$-components after the last query is $p+1\ge 2k+1$. It is easy to see that the number of balanced $q$-components
is at most $k-1$, since a balanced $q$-component must contain both $x_{ik}$ and $x_{ik+1}$ for some $i<k$. Thus
at least $k+2$ of the  $q$-components are unbalanced. It is also easy to see that the weight of any $q$-component
is at most $k+1$ (since $k$ vertices of the same color are always followed and preceded by $k$ vertices of the
opposite color, except for the first $k$ and last $k+1$ vertices). Now if one could show a majority vertex, the
weight of its $q$-component should be more than the sum of the weights of the other $q$-components, which is 
impossible, since the latter sum is at least $k+1$. 
This finishes the proof of the lower bound.

Next we prove the upper bound. Consider any coloring of the vertices of $P$ and let us denote the number of red (resp.\ blue) vertices among $\{ x_{1}, x_{2}, \ldots , x_{i}\}$
by $R(i)$ (resp.\ $B(i)$) and let us suppose  without loss of generality that $d:=R(n)-B(n) >0$ (notice that $n$ is odd). 
Observe that if $d\ge \sqrt{n}$, then by asking the first $n-\lceil \frac{d}2 \rceil$ edges 
(or any consecutive $n-\lceil \frac{d}2 \rceil$ edges) of $P$, we obtain a $q$-component of weight at least 
$\lceil \frac{d}2 \rceil +1$ and $\lceil \frac{d}2 \rceil -1$ $q$-components of weight (and also cardinality) 1. 
Thus any majority vertex of the large $q$-component is also a majority vertex of the whole graph, and we are done. Therefore, we 
might assume that $d < \sqrt{n}$. 

Let $D(i) := R(i)-B(i)$ (so $d=D(n)$), 
$\Delta:= \max_{i=1}^n |D(i)|$, and let $j$ be the smallest number, such that $|D(j)|=\Delta$. Since $d> 0$, we 
may suppose that $D(j) = \Delta$, otherwise we can reverse the order of the vertices and obtain a situation, where the similarly obtained
$D(j')$ is positive (the value $\Delta$ would be different then, but $d$ remains the same). Now we consider two 
cases, based on the value of $\Delta$. 

Case 1. $\Delta < 2\sqrt{n}$. Since all values $D(i)$ are in the interval $[-\Delta,\Delta]$, by the pigeonhole principle there must be a value $v$, for which $|\{ i : D(i)=v  \}| \ge \frac{n}{4\sqrt{n}+1} > \frac{\sqrt{n}}{5}$.  It is obvious that if $D(a) = D(b)$, then the number of red and blue vertices are the same in the subpath between the vertices $x_{a+1}$ and 
$x_b$. Let the elements of $\{ i : D(i)=v  \}$ be $i_1, i_2, \ldots ,i_r$ and let us query all edges of $P$, except the 
edges $(x_{i_1}, x_{i_1+1}),  (x_{i_2}, x_{i_2+1}), \ldots ,(x_{i_r}, x_{i_r+1})$. In this way we obtain $r$ or $r+1$ $q$-components, of which $r-1$ are balanced, therefore any majority vertex of the largest non-balanced 
$q$-component (which exists, since $n$ is odd) is a majority vertex of the whole graph as well. Since $r 
>\frac{\sqrt{n}}{5}$, we are done with this case.

Case 2. $\Delta \ge 2\sqrt{n}$. Recall that $j$ is the smallest number, such that $D(j)=\Delta$ and $d < \sqrt{n}$, i.e., the number of red vertices is $\Delta$ more than the number of blue vertices by $x_j$, but then the difference drops to $d$ by the end of $P$.
Thus there must exist a smallest number $k$, such that $k>j$ and $D(k) < \sqrt{n}$. Then the subpath $P'$ between the vertices $x_{j+1}$ and $x_k$ contains at least $\Delta-\sqrt{n} \ge \sqrt{n}$ more blue vertices, than red vertices.  Let now $j_1$ be the smallest index, such that the subpath of $P'$ between $x_{j+1}$ and $x_{j_1}$ contains exactly one more of the blue vertices than the red vertices.  Similarly, let $j_2$ be the smallest index, such 
that the subpath of $P'$ between $x_{j_{1}+1}$ and $x_{j_2}$ contains one more of blue vertices than red vertices, and so on.  It is clear that the indices $j_1,j_2, \ldots , j_{\lfloor \sqrt{n}\rfloor}$ are well-defined. Now 
let us query all edges of $P$, except the edges $(x_{j}, x_{j+1}),  (x_{j_1}, x_{j_1+1}), \ldots ,(x_{j_{\lfloor \sqrt{n}\rfloor}}, x_{j_{\lfloor \sqrt{n}\rfloor +1}})$. In this way we obtain $\lfloor \sqrt{n}\rfloor +2$ $q$-components, such that one of them has weight $\Delta$, $\lfloor \sqrt{n}\rfloor $ of them has 
weight 1, and one of them has weight smaller than $\lfloor \sqrt{n}\rfloor$, thus any majority vertex of the 
$q$-component of weight $\Delta$ is also a majority vertex of the whole graph, finishing the proof of Proposition \ref{nondet}.
\end{proof}

\subsection{Optimal graph with few edges}

In this subsection we prove Theorem \ref{minedge} which states that
for every $n$ there is a graph $G$ with $n$ vertices and $n(1+b(n))$ edges such that $m(G)=n-b(n)$.

\begin{proof}[Proof of Theorem \ref{minedge}]
Let $F_k$ be the graph obtained from a path $v_1v_2\dots v_k$ by adding $k$ vertices of degree 1, $u_1,u_2, \ldots, u_k$, to it such that $u_i$ is connected to $v_i$ for each $1 \le i \le k$. Let $k=\lfloor n/2 \rfloor$ and $G$ be the graph we obtain from $F_k$ by adding all the possible edges incident to any of the vertices $v_{k-b(n)+1}, \dots, v_{k}$.
\footnote{There are several non-isomorphic graphs $G'$ with $n(1+b(n))$ edges such that essentially the same proof shows $m(G')=n-b(n)$. 
For example we could take the union of a path on $n$ vertices with the biclique $K_{n-b(n),b(n)}$, but we have found our proof to be easier to present for the above graph $G$.}

We are going to define an algorithm $A_l$ for $l=2^i$. First we describe some properties of the algorithm. It uses the edges of $F_l$ and either gets back a DIFF answer at some point for an edge that connects two monochromatic $q$-components of the same size (which is a power of two), or shows that $F_l$ is monochromatic. Moreover, at any point it uses only the first $j$ vertices of the path $v_1\dots v_j$ and the leaves $u_1,\dots, u_j$ connected to them, for some $j$. Therefore, if there are vertices not appearing in any query, they form a connected graph.

We define algorithm $A_l$ recursively. Algorithm $A_1$ is trivial, it has only one query $u_1v_1$. Assume we have defined Algorithm $A_l$ and we are given $F_{2l}$. The graph $F_{2l}$ consists of two copies of $F_l$ and an additional edge, where the first copy has vertices $v_1,\dots, v_l,u_1,\dots,u_l$, the second copy has the remaining vertices, and the additional edge is $v_lv_{l+1}$. Algorithm $A_{l+1}$ runs algorithm $A_l$ separately for the first, and then for the second copy of $F_l$, and finally asks $(v_l,v_{l+1})$. If for either copy of $F_l$ we get back a DIFF answer at some point for an edge that connects two monochromatic $q$-components of the same size, we are done. Otherwise, both copies of $F_l$ are monochromatic. In this case a DIFF answer to the last query connects two monochromatic $q$-components of the same size, while a SAME answer shows $F_{2l}$ is monochromatic.

The algorithm showing $m(G)\le n-b(n)$ for even $n$ is based on an idea similar to the algorithm showing $m(K_n)\le n-b(n)$: we ask queries such that if the answer is DIFF, we obtain a balanced $q$-component (that we can discard), and otherwise we build larger and larger monochromatic $q$-components, where the number of vertices is a power of 2. In the following algorithm, whenever the answer is DIFF, we stop the current step and continue with the next.

We start Step 1 with the query $(v_{k},u_{k})$. Observe that these two vertices can be considered as the first vertices of a copy of $F_{2^i}$ having additionally the vertices $v_1,\dots, v_{2^{i-1}-1},u_1,\dots, u_{2^{i-1}-1}$, for every $i\le l=\lfloor \log n \rfloor$. Thus we run algorithm $A_l$. If we obtain a monochromatic $q$-component of size $2^l$, that is of the majority $q$-color, and we are done. If we obtain a DIFF answer, we continue with Step 2, which starts with asking $v_{k-1}u_{k-1}$. Observe that any consecutive part of $F_k$ together with $u_k$ and $v_k$ forms a copy of $F_l$ for some $l$. More precisely, there is a copy of $F_l$ on the vertices $v_j,\dots,v_{j+l-2},u_j,\dots,u_{j+l-2},v_{k-1},u_{k-1}$, provided $j+l-2<k-1$. 
We continue with the vertex $v_j$ if $v_{j-1}$ has the largest index among those appearing in any query in Step 1. Similarly to Step 1, we run algorithm $A_l$ for the largest $l$ possible, i.e., the largest $l$ that is a power of 2 and is smaller than $k-j+1$. 

In general, for Step $i$ we take $v_{k-i+1}$, the first vertex $v_j$ not appearing in any query in Step $i-1$, and $2^r-2$ consecutive vertices next to it for the largest $r$ possible without arriving back to $v_{k-i+1}$. Furthermore, we take $u_m$ for every $v_m$ we took. More precisely, we take $v_j,\dots, v_{j+2^r-2},u_j,\dots, u_{j+2^r-2}$, $v_{k-i+1}$ and $u_{k-i+1}$.
This is a copy of $F_{2^r}$, thus we can run Algorithm $A_{2^r}$ on it. If we obtain a monochromatic $q$-component, that is of the majority color, and we are done. Indeed, we took the largest $r$ possible, thus more than half of the vertices not appearing in any query before Step $i$ are of that color. The vertices appearing in earlier queries are balanced, as the $q$-components obtained in earlier steps are balanced.

Observe that if all the steps end with DIFF answers, we had at least $b(n)$ steps before the end of the algorithm, i.e. before every vertex appeared in a query. Indeed, all the components have order that is a power of 2, and $n$ cannot be written as the sum of less than $b(n)$ powers of 2.
After Step $b(n)$, the remaining vertices (if there are any) form a connected graph, thus we can ask a spanning tree of that graph to find a majority vertex there. Altogether the query graph is a forest of at least $b(n)$ $q$-components, thus at most $n-b(n)$ queries were asked.
\end{proof}



\section{Questions}
We collect below the most important questions that remain open.

\begin{itemize}
    \item What is the complexity of computing $m(\underline w)$ and $m(G)$?
    
    \item For which $\underline{w}'=(2w_1,w_3,w_4,\dots,w_k)$ does $m(\underline w')=m(\underline w)-1$ hold?
    
    \item Does $m(T)\ge n-3$ hold for every tree or path on $n$ vertices?
    
    \item What is the least number of edges a graph on $n$ vertices can have if $m(G)=n-b(n)$?
\end{itemize}

\subsubsection*{Acknowledgment}
We would like to thank our anonymous reviewers for several useful suggestions that improved the presentation of our paper.

\end{document}